\documentclass[psamsfonts]{amsart} 
\usepackage{amssymb}
\usepackage[dvips]{graphicx}

\usepackage{graphicx}
\usepackage{amssymb}                       
\usepackage{amsmath}
\usepackage{color}
\usepackage{times}

\usepackage[unicode,bookmarks,colorlinks]{hyperref}
\hypersetup{
}

\definecolor{mahogany}{cmyk}{0, 0.77, 0.87, 0}
\definecolor{salmon}{cmyk}{0, 0.53, 0.38, 0}
\definecolor{melon}{cmyk}{0, 0.46, 0.50, 0}
\definecolor{yellowgreen}{cmyk}{0.44, 0, 0.74, 0}
\definecolor{brickred}{cmyk}{0, 0.89, 0.94, 0.28}
\definecolor{OliveGreen}{cmyk}{0.64, 0, 0.95, 0.40}
\definecolor{RawSienna}{cmyk}{0, 0.72, 1.0, 0.45}
\definecolor{ZurichRed}{rgb}{1, 0, 0} % Red of svgnames

\usepackage{fancyhdr}
\usepackage{amsmath,amstext,amssymb,amsopn,amsthm}
\usepackage{amsmath,amssymb,amsthm}
\usepackage[mathscr]{eucal}

\newtheorem{conjecture}{Conjecture}

\reversemarginpar
\definecolor{RB}{rgb}{0.7,0.1,0.2}

\definecolor{MP}{rgb}{0.1,0.5,0.1}

\begin{document}

%\newtheorem{thm}{Theorem}
%\numberwithin{thm}{Theorem}
\newtheorem{lemma}[thm]{Lemma}
\newtheorem{corr}[thm]{Corollary}
%\numberwithin{corollary}[thm]{section}
\newtheorem{proposition}{Proposition}
\newtheorem{theorem}{Theorem}[section]
\newtheorem{deff}[thm]{Definition}
\newtheorem{case}{Case}
%\numberwithin{deff}{section}
\newtheorem{prop}[thm]{Proposition}
%\numberwithin{equation}{subsection}
%\numberwithin{equation}{section}
\newtheorem{example}{Example}

\newtheorem{corollary}{Corollary}

\theoremstyle{definition}
\newtheorem{remark}{Remark}

\numberwithin{equation}{section}
\numberwithin{definition}{section}
%\numberwithin{problem}{section}
\numberwithin{corollary}{section}
%\numberwithin{proposition}{subsection}

\numberwithin{theorem}{section}

\numberwithin{remark}{section}
\numberwithin{example}{section}
\numberwithin{proposition}{section}

\newcommand{\calD}{\bD}

\newcommand{\conjugate}[1]{\overline{#1}}
\newcommand{\abs}[1]{\left| #1 \right|}
\newcommand{\cl}[1]{\overline{#1}}
\newcommand{\expr}[1]{\left( #1 \right)}
\newcommand{\set}[1]{\left\{ #1 \right\}}

\newcommand{\calC}{\mathcal{C}}
\newcommand{\calK}{\mathcal{K}}
\newcommand{\calS}{\mathcal{S}}
\newcommand{\calE}{\mathcal{E}}
\newcommand{\calF}{\mathcal{F}}
\newcommand{\Rd}{\mathbb{R}^d}
\newcommand{\BR}{\bD(\Rd)}
\newcommand{\R}{\mathbb{R}}
\newcommand{\al}{\alpha}
\newcommand{\RR}[1]{\mathbb{#1}}
\newcommand{\bR}{\mathrm{I\! R\!}}
\newcommand{\ga}{\gamma}
\newcommand{\om}{\omega}
\newcommand{\A}{\mathbb{A}}
\newcommand{\bH}{\mathbb{H}}

\newcommand{\bb}[1]{\mathbb{#1}}
\newcommand{\bI}{\bb{I}}
\newcommand{\bN}{\bb{N}}

\newcommand{\uS}{\mathbb{S}}
\newcommand{\M}{{\mathcal{M}}}
\newcommand{\calB}{{\mathcal{B}}}

\newcommand{\W}{{\mathcal{W}}}

\newcommand{\m}{{\mathcal{m}}}

\newcommand {\mac}[1] { \mathbb{#1} }

\newcommand{\bD}{\mathbb{D}}

\newcommand{\bC}{\Bbb C}

\newtheorem{rem}[theorem]{Remark}
\newtheorem{dfn}[theorem]{Definition}
\theoremstyle{definition}
\newtheorem{ex}[theorem]{Example}
\numberwithin{equation}{section}

\newcommand{\Pro}{\mathbb{P}}
\newcommand\F{\mathcal{F}}
\newcommand\E{\mathbb{E}}
\newcommand\e{\varepsilon}
\def\H{\mathcal{H}}
\def\t{\tau}

\newcommand{\blankbox}[2]{%
  \parbox{\columnwidth}{\centering
%    Set fboxsep to 0 so that the actual size of the box will match the
%    given measurements more closely.
    \setlength{\fboxsep}{0pt}%
    \fbox{\raisebox{0pt}[#2]{\hspace{#1}}}%
  }%
}

\title[method of rotations for L\'evy Multipliers]{A method of rotations for L\'evy Multipliers }

\author{Michael Perlmutter*}\thanks{* Supported in part  by NSF Grant
\#1403417-DMS  under PI Rodrigo Ba\~nuelos}
\address{Department of Mathematics, Purdue University, West Lafayette, IN 47907, USA}
\email{mperlmut@math.purdue.edu}
%\date{}
\maketitle

\begin{abstract} We use a method of rotations to study the $L^p$ boundedness, $1<p<\infty$, of Fourier multipliers which arise as the projection of martingale transforms with respect to symmetric $\alpha$-stable processes, $0<\alpha<2$. Our proof does not use the fact that $0<\alpha<2$, and therefore allows us to obtain a larger class of multipliers which are bounded on $L^p$. As in the case of the multipliers which arise as the projection of martingale transforms, these new multipliers also have potential applications to the study of the $L^p$ boundedness of the Beurling-Ahlfors transform; see conjecture \ref{Conj} below.   
\end{abstract}

 \section{Introduction and statement of Results}
The Beurling-Ahlfors transform, defined on the complex plane by 
\begin{equation*}
Bf(z) = -\frac{1}{\pi}p.v.\int_{\mathbb{C}} \frac{f(w)}{(z-w)^2} dw
\end{equation*}
for $f \in C^\infty_0(\mathbb{C})$,
 is the analogue of the Hilbert transform on the real line.  It is a Calder\'on-Zygmund singular integral operator, and it is a  Fourier multiplier with 
\begin{equation*}
\widehat{Bf}(\xi) = \frac{\bar{\xi}}{\xi} \widehat{f}(\xi).
\end{equation*}
%Therefore, it has the important property of interchanging the holomorphic and anti-holomorphic parts of a smooth compactly-supported function on the complex plane, that is 
%\begin{equation}
%\partial_z Bf = \partial_{\bar{z}} f CHECK
%\end{equation}
%for all $f\in C^\infty_0(\mathbb{C})$. Because of this property, the Beurling-Ahlfors transform plays an important role in many areas of analysis including non-linear PDE and the theory of quasi-conformal maps.
%%
The classical theory of Calder\'on-Zygmund singular integrals shows that $B$ is bounded on $L^p(\mathbb{C})$ for $1<p<\infty$.  Because of its many connections to quasiconformal mappings and other problems in complex analysis (see for example \cite{AstIwaMar})  there has been a lot of interest for many years in finding its operator norm on  $L^p(\mathbb{C}),$ $1<p<\infty$, which we denote $\|B\|_p$.   In \cite{Leh}, Lehto showed that $\|B\|_p \geq (p^*-1)$, where $p^* = \max\{p,\frac{p}{p-1}\}$. A long standing conjecture of Iwaniec \cite{Iwa} is that $\|B\|_p = (p^*-1)$.  The literature on this subject is now quite large, and it would be impossible for us to  review it here in its entirety.  For some of this literature, we refer the reader to the overview article \cite{Ban1} and the many references given there. 

Despite the efforts of many researchers, Iwaniec's conjecture remains open. There are, however, many partial results, and the techniques developed in these efforts have lead to many other interesting questions and applications.   In particular, there are a number of probabilistic constructions which provide upper bounds for $\|B\|_p$. The primary purpose of this paper is to study the $L^p$ boundedness of operators closely related to one of these constructions.

  In \cite{BanWan}, Ba\~nuelos and Wang used the background radiation process of Gundy and Varopolous \cite{GunVar} combined with Burkholder's inequalities regarding the sharp $L^p$ bounds of martingale transforms \cite{Bur3} to show that $\|B\|_p\leq 4(p^*-1)$. This result, in addition to being, at the time, the best known upper bound for $\|B\|_p$, had the desirable property that it directly involved the constant  $p^*-1$. This property is shared by many estimates which are obtained by probabilistic methods.  
 In \cite{NazVol}, Nazarov and Volberg showed that $\|B\|_p \leq 2(p^*-1)$ using Bellman function techniques to prove a Littlewood-Paley inequality for heat extensions. The bound $\|B\|_p \leq 2(p^*-1)$ was again obtained in \cite{BanMen} using a method that is similar to \cite{BanWan} but which  replaces the background radition process with space-time Brownian motion. The methods of \cite{BanMen} were refined in \cite{BanJan} taking advantage of the fact that the martingales arising in the representation of the Beurling-Ahlfors transform have certain orthogonality properties to produce the bound $\|B\|_p \leq 1.575(p^*-1)$ which is, as of now, the best known bound valid for all $1<p<\infty$. In \cite{BorJanVol}, this bound was improved   to $\|B\|_p \leq 1.4(p^*-1)$  for $1000<p<\infty.$ 
 
 The method used in \cite{BanMen} and later in \cite{BanJan} and \cite{BorJanVol} is to embed $L^p(\mathbb{R}^n)$ into a space of $p-$integrable martingales via composition of a space-time Brownian motion with caloric functions,  apply a martingale transform, and then project back to $L^p(\mathbb{R}^n)$ using conditional  expectation.  This yields a large class of Fourier multipliers that includes the Beurling-Ahlfors transform with $L^p$ bounds that are multiples of $p^*-1$.

In \cite{BanBog1} and \cite{BanBog2}, it was shown that interesting Fourier multipliers can also be obtained by considering the conditional expectation of martingale transforms involving more general L\'evy processes in place of Brownian motion. In particular, in \cite{BanBog2}, using the symmetric $\alpha-$stable process, $0<\alpha<2$, and Burkholder's sharp martingale transform inequalities,   it is shown that for all $\varphi\in L^\infty(\mathbb{S}^{n-1})$, $\|\varphi\|_\infty\leq 1$, the operator defined by $\widehat{T_{m_\alpha}f}(\xi) = m_{\alpha}(\xi)\widehat{f}(\xi)$ where

\begin{equation*}
m_\alpha(\xi) = \frac{\int_{\mathbb{S}^{n-1}} |\xi\cdot\theta|^\alpha \varphi(\theta)d\sigma(\theta)}{\int_{\mathbb{S}^{n-1}} |\xi\cdot\theta|^\alpha d\sigma(\theta)},
\end{equation*}

\noindent is bounded on $L^p(\mathbb{R}^n)$, $1<p<\infty,$ with 
\begin{equation}
\|T_{m_\alpha} f\|_p \leq (p^*-1) \|f\|_p\quad\text{for all }f\in L^p(\mathbb{R}^n).
\end{equation}

When $n=2$, the  choice  of $\varphi(\theta)=e^{2i\arg(\theta)}$ yields 
\begin{equation}
\label{BeurRep} m_{\alpha}(\xi) = \frac{\alpha}{\alpha+2} \frac{\bar{\xi}}{\xi} 
\end{equation}
and therefore 
$$
T_{m_\alpha}f=\frac{\alpha}{\alpha+2} Bf.  
$$
Letting $\alpha\nearrow 2$, we recover the estimate $\|B\|_p\leq 2(p^*-1)$ which was proved in \cite{NazVol} and \cite{BanMen}. 

The formula (\ref{BeurRep}) does not depend on the fact that $0<\alpha<2$.  That is, for all $r>0$, 

\begin{equation*}
m_r(\xi) = \frac{\int_{\mathbb{S}^1} |\xi\cdot\theta|^r e^{2i\arg(\theta)}d\sigma(\theta)}{\int_{\mathbb{S}^1} |\xi\cdot\theta|^r d\sigma(\theta)}=\frac{r}{r+2} \frac{\bar{\xi}}{\xi}.
\end{equation*}
In  fact it is clear that for any $r>0$ and any $\varphi\in L^\infty(\mathbb{S}^{n-1})$ 

\begin{equation}
\label{m} m_r(\xi) = \frac{\int_{\mathbb{S}^{n-1}} |\xi\cdot\theta|^r\varphi(\theta) d\sigma(\theta)}{\int_{\mathbb{S}^{n-1}} |\xi\cdot\theta|^r d\sigma(\theta)}
\end{equation}
gives rise to a Fourier multiplier which is bounded on $L^2(\mathbb{R}^n)$. However, for $r>2$,  it is unknown if this multiplier arises from martingale transforms of any kind (see remark \ref{McCo}) and its boundedness on $L^p(\R^n)$ for any $p\not =2$ is by itself an interesting problem.  This motivated the following conjecture which appeared in \cite{Ban1}.

\begin{conjecture}\label{Conj} Let $n\geq 2$, $0<r<\infty$, $\varphi\in L^\infty(\mathbb{S}^{n-1})$, $\|\varphi\|_\infty\leq 1$, and let $m_r$ be defined as in (\ref{m}). Then the corresponding operator, $T_{m_r}$, is bounded on $L^p(\mathbb{R}^n)$ for all $1<p<\infty$ and
\begin{equation*}
\|T_{m_r} f\|_p \leq (p^*-1) \|f\|_p,  \quad\text{for all } f\in L^p(\mathbb{R}^n).
\end{equation*}
\end{conjecture}

This is a very strong conjecture, which if true would imply Iwaniec's conjecture  \cite{Iwa} that $\|B\|_p\leq p^*-1$. The main results of this paper concern the  boundedness of these multipliers on $L^p(\R^n)$ for all $1<p<\infty$ with some information on the constant. More precisely, we prove the following two theorems.

\begin{theorem}
\label{main} Let $n\geq 2$, $0<r<\infty$, $\varphi\in L^\infty(\mathbb{S}^{n-1})$, $\|\varphi\|_{\infty}\leq1$, and let $m_r$ be defined as in (\ref{m}). Then the corresponding operator, $T_{m_r}$, is bounded on $L^p(\mathbb{R}^n)$ for all $1<p<\infty$ and
\begin{equation*}
\|T_{m_r}f\|_p \leq C_n(p^*-1)^{6n}\frac{\Gamma(\frac{r+n}{2})}{\Gamma(\frac{r+1}{2})}\|f\|_p, \quad\text{for all }f\in L^p(\mathbb{R}^n), 
\end{equation*}
where $C_{n}$ is a constant  which depends only on $n.$
\end{theorem}

\begin{remark}
Sterling's formula implies that if $a>0$
\begin{equation*}
\frac{\Gamma(x+a)}{\Gamma(x)} = O(x^a) \quad\text{as }x\rightarrow\infty.
\end{equation*}
Therefore,
\begin{equation*} 
 \frac{\Gamma(\frac{r+n}{2})}{\Gamma(\frac{r+1}{2})}
%&=C_n \frac{\Gamma(\frac{r+1}{2} + \frac{n-1}{2})}{\Gamma(\frac{r+1}{2})}\\
= O(r^{(n-1)/2}) \quad \text{as } r\rightarrow\infty.
\end{equation*}

\end{remark}

In the case that $r$ is sufficiently large, we can use the H\"{o}rmander-Mikhlin multiplier theorem to obtain estimates on the $L^p$ bounds of $T_{m_r}$ that are linear in $p$ as $p\rightarrow \infty$. 

\begin{theorem}

\label{second}Let $n\geq 2$ and define $n_0=\lfloor\frac{n}{2}\rfloor+1$.  Let  $n_0\leq r<\infty$, $\varphi\in L^\infty(\mathbb{S}^{n-1})$, $\|\varphi\|_{\infty}\leq1$, and let $m_r$ be defined as in (\ref{m}). Then the corresponding operator, $T_{m_r}$, is bounded on $L^p(\mathbb{R}^n)$ for all $1<p<\infty$ and
\begin{equation*}\|T_{m_r}f\|_p\leq C_n \max\{r^{n_0},1\} (p^*-1)\|f\|_p,  \quad\text{for all }f\in L^p(\mathbb{R}^n),
\end{equation*}
where $C_n$ is a constant depending only on $n$.
Furthermore, $T_{m_r}$ is weak-type $(1,1)$ and 
\begin{equation*}
|\{T_{m_r}f(x)>\lambda\}| \leq C_n\max\{r^{n_0},1\}\frac{\|f\|_1}{\lambda}.
\end{equation*}

\end{theorem}
\begin{remark}
Comparing the estimates in theorem \ref{main} and theorem \ref{second}, we see that each has some advantages over the other. The constants obtained in theorem \ref{main} have slower growth as $r\rightarrow \infty$ than those obtained in theorem \ref{second} and have the advantage of being valid for all $r>0$. On the other hand, theorem \ref{second} gives  estimates which are linear in $p$ as $p\rightarrow \infty$ and includes weak-type (1,1) estimates which theorem \ref{main} does not. This is because the proof of theorem \ref{main} involves the method of rotations and the Marcinkiewicz multiplier theorem, neither of which give weak-type inequalities. We also remark that it is unknown if the operators which are obtained in \cite{BanBog1} and \cite{BanBog2} satisfy weak-type (1,1) inequalities. While it is true that martingale transforms do satisfy weak-type (1,1) estimates, these estimates are not preserved under conditional expectation. Weak-type estimates for the operators constructed in \cite{BanMen} and \cite{BanWan} were proved  in \cite{Perl} using the Calder\'on--Zygmund theory.  
\end{remark}

The rest of this paper is organized as follows. In section 2, we will give background information regarding L\'evy processes and their use in studying Fourier multipliers. In sections 3  and 4 we will prove theorems \ref{main} and \ref{second}  respectively. Lastly, in section 5 we will demonstrate how the methods use to prove theorem \ref{main} may be used to study other L\'evy multipliers.   Throughout this paper we will use the following notation. If $m(\xi)$ is a bounded complex-valued function on $\mathbb{R}^n$, $T_m$ shall refer to the operator on $L^2(\mathbb{R}^n)$ defined by $\widehat{T_mf}(\xi) = m(\xi)\widehat{f}(\xi)$. If, for some $1<p<\infty$, $T_m$ admits a bounded extension to $L^p(\mathbb{R}^n)$, than we shall say that $m$ is an $L^p$ multiplier. We shall also assume that $n\geq 2$ for the rest of the paper.

\section{Background}

A L\'evy process on $\mathbb{R}^n$ is an $\mathbb{R}^n-$valued stochastic process, $(X_t)_{t\geq 0}$, which almost surely starts at the origin, has stationary, independent increments, and satisfies the stochastic continuity condition $\lim_{t \searrow 0} \mathbb{P}(|X|_t > \epsilon) = 0$ for all  $\epsilon >0$. The famous L\'evy-Khintchine formula states that if $X_t$ is any L\'evy process, there exists a point $b\in \mathbb{R}^n$, a non-negative symmetric $n\times n$ matrix $B$, and a measure $\nu$ such that $\nu(\{0\})=0$ and $$\int_{\mathbb{R}^n} \min\{|z|^2,1\} d\nu(z) < \infty,$$ such that the characteristic function of $X_t$ is given by $\mathbb{E}(e^{i\xi \cdot X_t}) = e^{t \rho(\xi)}$ where
\begin{equation*}
\rho(\xi) = i b\cdot \xi -\frac{1}{2}B\xi\cdot\xi + \int_{\mathbb{R}^n}\left[ e^{i\xi\cdot z} -1 -i(\xi\cdot z)\mathbb{I}_{(|z|<1)}\right] \nu(dz).
\end{equation*}
$(b,B,\nu)$ is referred to as the L\'evy triple of $X_t$. The triple $(b,0,0)$ corresponds to a drift process $X_t = bt$; $(0,B,0)$ corresponds to a centered Gaussian process with whose covariance is given by $[X^{i}_s,X^{j}_t] = b_{i,j}\min\{s,t\}$; and $(0,0,\nu)$ corresponds to a ``pure-jump'' process. In particular, $(0,I,0)$ corresponds to standard Brownian motion with $\rho(\xi)= -\frac{1}{2}|\xi|^2$, and for $0<\alpha<2$, $(0,0,d\nu(z) = C_{n,\alpha}\frac{1}{|z|^{n+\alpha}}dz)$ corresponds the symmetric $\alpha$-stable process with $\rho(\xi)= -|\xi|^\alpha$. If $X_t$ and $Y_t$ are independent L\'evy processes with triples $(b_X,B_X,\nu_X)$ and $(b_Y,B_Y,\nu_Y)$, then $X_t+Y_t$ is a L\'evy process with the triple $(b_X+b_Y,B_X+B_Y,\nu_X+\nu_Y)$. Therefore, the L\'evy-Khinchtine formula says that any L\'evy process can be decomposed into the sum of three independent L\'evy processes, a drift process, a centered Gaussian process, and a pure jump process.

L\'evy processes have been extensively used to study the $L^p$ boundedness of Fourier multipliers and in particular certain Calder\'on-Zygmund singular integrals. In this section, we will present a summary of two constructions which can be used to study the Beurling-Ahlfors transform. For further details of these two constructions we refer the reader to \cite{BanMen} and to \cite{BanBog1}, \cite{BanBog2} respectively. For examples of how related methods have been used to study other operators, we refer the reader to \cite{AppBan}, \cite{Ban1}, and \cite{Ban2}. In all of these cases, the method is based on the same fundamental idea. For a function $f$ in $L^p(\mathbb{R}^n)$, we construct a martingale $M(f)_t$ such that $\sup_{t}\|M(f)_t\|_p = \|f\|_p$. Then we apply a martingale transformation to get a new martingale, $N(f)_t$, such that $\sup_{t}\|N(f)_t\|_p \leq C_p \sup_{t}\|M(f)_t\|_p$. Finally, we project $N(f)_t$ onto $L^p(\mathbb{R}^n)$ using conditional expectation to get a new function which we denote by $Sf(x)$. Conditional expectation is a contraction on $L^p(\mathbb{R}^n)$ so $\|Sf\|_p \leq \sup_{t}\|N(f)\|_p$. Combining these three inequalities yields $\|Sf\|_p\leq C_p \|f\|_p$. If appropriate choices are made at each step, this operator  will coincide with an operator of classical interest in analysis such as the Beurling-Ahlfors transform.

In \cite{BanMen}, this procedure was carried out using martingales involving space-time Brownian motion. For $f\in L^p(\mathbb{R}^n)$, we consider $V_f(x,t)=\mathbb{E}(f(B_t+x)|B_0=0)=(p_t \ast f) (x)$, where $p_t(x)=\frac{1}{(2\pi t)^{n/2}}e^{-|x|^2/2t}$ is the heat kernel for the half Laplacian and $B_t$ is Brownian motion in $\mathbb{R}^n$ with initial distribution given by the Lebesgue measure. For fixed $T>0$, the process $(Z_t)_{0\leq t\leq T}=(B_t,T-t)_{0\leq t\leq T}$ is called space-time Brownian motion. It\^{o}'s formula shows that $V_f(Z_t)_{0\leq t\leq T}$ is a martingale
and  
\begin{equation*}
V_f(Z_t)-V_f(Z_0)= \int_0^t \nabla_x V_f(Z_s)\cdot dB_s.
\end{equation*}
Furthermore, 
\begin{align*}
\|V_f(Z_t)\|^p_p &= \mathbb{E}|V_f(Z_t)|^p = \int_{\mathbb{R}^n} \mathbb{E}^x|V_f(Z_t)|^pdx \\
&=\int_{\mathbb{R}^n} \int_{\mathbb{R}^n} p_t(x-y) |f(y)|^p dydx = \int_{\mathbb{R}^n} |f(y)|^p dy = \|f\|_p^p.
\end{align*}

\noindent For any $n\times n$ matrix-valued function, $A(s)$, $ s >0$ such that 
\begin{equation*}
\|A\|=\sup_{s}\sup_{|v|\leq 1}\{|A(s)v|\} \leq 1,
\end{equation*}
we define the martingale transform of $V_f(Z_t)$ by $A(s)$ as
\begin{equation*}
A \star V_f(Z_t) = \int_0^t A(s)\nabla_x V_f(Z_s)\cdot dB_s.
\end{equation*}

The quadratic variations of $V_f(Z_t)$ and $A\star V_f(Z_t)$ are given by 
\begin{equation*}
[V_f(Z)]_t = \int_0^t |\nabla_x  V_f(Z_s)|^2 ds \quad\text{and}\quad
[A\star V_f(Z)]_t = \int_0^t |A(s)\nabla_x  V_f(Z_s)|^2 ds.
\end{equation*}
Since $\|A\|\leq 1$,
$A\star V_f(Z_t)$ is differentially subordinate to $V_f(Z_t)$, that is, the process $[V_f(Z)]_t-[A\star V_f(Z)]_t$ is non-decreasing. Therefore, by  Burkholder's celebrated theorem (see \cite{Bur3}) we have that
\begin{equation*}
\sup_t\|A\star V_f(Z_t)\|_p \leq (p^*-1) \sup_t\|V_f(Z_t)\|_p = (p^*-1)\|f\|_p.
\end{equation*}
 To project $A\star V_f(Z_t)$ back onto $L^p(\mathbb{R}^n)$ we define
\begin{equation*}
S_A^Tf(x) = E\left(\int_0^T A(T-s)\nabla_x V_f(Z_s)\cdot dB_s | B_T=x\right).
\end{equation*}
$S_A^T$ is a bounded linear operator on $L^p(\mathbb{R}^n)$ with 
\begin{equation*}
\|S_A^Tf(x)\|_p \leq (p^*-1)\|f\|_p.
\end{equation*}
Moreover, $S_A^T$ is a Fourier multiplier with 
\begin{equation*}
\widehat{S_A^Tf}(\xi) = \left(4\pi^2\int_0^T A(s)\xi\cdot \xi e^{-4\pi^2s|\xi|^2} ds\right) \widehat{f}(\xi).
\end{equation*}
Letting $T\rightarrow \infty$, we see that the limiting operator defined by 
\begin{equation}
\label{heatmult}\widehat{S_Af}(\xi) = \left(4\pi^2\int_0^\infty A(s)\xi\cdot \xi e^{-4\pi^2s|\xi|^2} ds\right) \widehat{f}(\xi)
\end{equation} 
is bounded on $L^p$  and
 \begin{equation}
\label{spacetimebound}\|S_Af(x)\|_p \leq (p^*-1)\|f\|_p.
\end{equation}
If $A(s)=A$ is constant, we can evaluate the integral in (\ref{heatmult}) and see that 
\begin{equation}
\label{constA} 4\pi^2\int_0^\infty A\xi\cdot \xi e^{-4\pi^2s|\xi|^2} ds = \frac{-\frac{1}{2}A\xi\cdot\xi}{-\frac{1}{2}|\xi|^2}.
\end{equation}
Recalling that for Brownian motion the L\'evy exponent is given by $\rho(\xi) = -\frac{1}{2}|\xi|^2$, we can interpret this multiplier as a ``modulation'' of the L\'evy exponent divided by the ``unmodulated'' L\'evy exponent.  If we choose 

\begin{equation*}
A = \frac{1}{2}\left(\begin{array}{lr}
1&i\\
i&-1
\end{array}\right),
\end{equation*}
we see that $S_Af(x)=\frac{1}{2}Bf(x)$. Combining this with (\ref{spacetimebound}) yields the inequality  
\begin{equation}
\label{twobound} \|Bf\|_p \leq 2(p^*-1)\|f\|_p
\end{equation}
which was mentioned in the introduction.

%Other than the use of It\={o}'s formula to write $V_f(Z_t)$ as a stochastic integral, the above construction does not seem to heavily depend on the structure of Brownian motion. Therefore, it is natural to ask if we can  construct Fourier multipliers as the projections of  martingale transforms with respect to other L\'evy processes? If so, can these multipliers be used to study the Beurling-Ahlfors transform? 

In \cite{BanBog1} and \cite{BanBog2}, this construction was generalized by replacing Brownian motion with more general L\'evy processes.  Let $\nu$ be a L\'evy measure on $\mathbb{R}^n$, $\varphi$ a complex-valued function on $\mathbb{R}^n$ with $\|\varphi\|_\infty\leq 1$, let $\mu$ a finite Borel measure on $\mathbb{S}^{n-1}$, and $\psi$ a complex-valued function on $\mathbb{S}^{n-1}$ with $\|\psi\|_\infty\leq 1.$ Define $m_{\mu,\nu}(\xi)$ by
\begin{equation}
\label{Levym}m_{\mu,\nu}(\xi) = \frac{\int_{\mathbb{R}^n} (\cos(\xi \cdot z)-1)\varphi(z)\nu(dz)+A\xi\cdot\xi}{\int_{\mathbb{R}^n} (\cos(\xi \cdot z)-1)\nu(dz)+B\xi\cdot\xi}
\end{equation}
where 
\begin{equation*}
A = \left(\int_{\mathbb{S}^{n-1}} \theta_i\theta_j \psi(\theta) d\mu(\theta)\right)_{1\leq i,j\leq n} \quad\text{and}\quad B = \left(\int_{\mathbb{S}^{n-1}} \theta_i\theta_j d\mu(\theta)\right)_{1\leq i,j\leq n}.
\end{equation*}
Note that $(\cos(\xi\cdot z)-1) = \Re(e^{i\xi\cdot z}-1-i(\xi\cdot z)\mathbb{I}_{(|z|<1)}$). Therefore, similarly to (\ref{constA}), $m_{\mu,\nu}$ may be interpreted as a ``modulation'' of the real part of the L\'evy exponent of some process, $X_t$, divided by the ``unmodulated'' real part of the L\'evy exponent of $X_t$. The primary result of \cite{BanBog2} is to show that $m_{\mu,\nu}$   is an $L^p$ multiplier for all $1<p<\infty$ and
\begin{equation*}
\|T_{m_{\mu,\nu}}f\|_p\leq (p^*-1)\|f\|_p \quad\text{for all }f\in L^p(\mathbb{R}^n).
\end{equation*}

We will now give a brief summary of how this multiplier is obtained in the case where $\mu=0$ and $\nu$ is symmetric and finite, which corresponds to $X_t$ being a compound Poisson process. (The general case can then be proved by symmetrization and approximation arguments, see \cite{BanBog2} for details.) 
 Similarly to \cite{BanMen}, we fix $T>0$, let $(Z_t)_{0\leq t\leq T} = (X_t,T-t)_{0\leq t\leq T}$, and let $V_f(x,t) = P_tf(x) = \mathbb{E}^T(f(X_t+x))$. It is shown in \cite{BanBog1} that $V_f(Z_t)$ is a martingale, with $\sup_{t}\|V_f(Z_t)\|_p=\|f\|_p$ for all $1<p<\infty$, and by the generalized It\^{o}'s formula (see for example \cite{Protter})

\begin{equation*}
V_f(Z_t)-V_f(Z_0) = \int_0^{t+}\int_{\mathbb{R}^n} [V_f(Z_{s-}+z)-V_f(Z_{s-})] \tilde{N}(ds,dz),
\end{equation*}
where $Z_{s-} = \lim_{u\nearrow s}Z_u$, and $\tilde{N}$ is the so called compensator, defined for each fixed $t>0$ on Borel sets of $ \mathbb{R}^n$ by 
\begin{equation*}
\tilde{N}(t,A) = N(t,A) -t\nu(A) 
\end{equation*}
where $N$ is a Poisson random measure that descibes the jumps of $X_t$, i.e. 
\begin{equation*}
N(t,A) = |\{ s: 0\leq s\leq t, X_{s}-X_{s-}\in A\}|.
\end{equation*}

Therefore if $\varphi:\mathbb{R}^n\rightarrow \mathbb{C}$ with $\|\varphi\|_\infty\leq 1$, we can define the martingale transform of $V_f(Z_t)$ by $\varphi$ as 
\begin{align*}
\varphi\star V_f(Z_t)= \int_0^{t+}\int_{\mathbb{R}^n} [V_f(Z_{s-}+z)-V_f(Z_{s-})]\varphi(z) \tilde{N}(ds,dz).
%\varphi\star V_f(Z_t) &= \sum_{0<s\leq t} (V_f(X_s,T-s)-V_f(X_{s-},T-s))\varphi(\Delta(X_s))\\ &\quad\: - \int_0^t\int_{\mathbb{R}^n} (V_f(X_{s-}+z,T-s)-V_f(X_{s-},T-s))\varphi(z) d\nu(z) ds.
\end{align*}
The quadratic variations of $V_f(Z_t)$ and $\varphi \star V_f(Z_t)$ are given by
\begin{equation*}
[V_f(Z)]_t = \int_0^{t+}\int_{\mathbb{R}^n} |V_f(Z_{s-}+z)-V_f(Z_{s-})|^2 N(ds,dz)
\end{equation*}
\noindent and
\begin{equation*}
[\varphi\star V_f(Z)]_t = \int_0^{t+}\int_{\mathbb{R}^n} |V_f(Z_{s-}+z)-V_f(Z_{s-})|^2 |\varphi(z)|^2N(ds,dz).
\end{equation*}
Therefore, $\varphi\star V_f(Z_t)$ is differentially subordinate to $V_f(Z_t)$ and 
\begin{equation*}
\sup_t\|\varphi\star V_f(Z_t)\|_p \leq (p^*-1)\|f\|_p.
\end{equation*}

A projection operator can be defined by
\begin{equation*}
S^T_\varphi f(x) = \mathbb{E}^T(\varphi\star V_f(Z_T)|Z_T=(x,0))
\end{equation*}
and we again have that 
\begin{equation*}
\label{Tbound}\|S^T_\varphi f(x)\|_p \leq (p^*-1)\|f\|_p.
\end{equation*}
It is shown that as $T\rightarrow\infty$, a limiting operator, $S_\varphi$, exists and satisfies the bound 
\begin{equation*}
\|S_\varphi f(x)\|_p \leq (p^*-1)\|f\|_p.
\end{equation*}
Moreover, $S_\varphi$ is a Fourier multiplier and $\widehat{S_\varphi f}(\xi) = m_{\mu,\nu}(\xi)\widehat{f}(\xi)$.

A particularly interesting class of operators occurs when we take $X_t$ to be the rotationally symmetric $\alpha$-stable process with $0<\alpha<2$ and assume that $\varphi$ is homogeneous of order zero. In polar coordinates, we may write $d\nu(z) = C_{n,\alpha}r^{-1-\alpha}drd\sigma(\theta)$ where $C_{n,\alpha}$ is a constant chosen so that 
\begin{equation*}
\rho(\xi) = \int_{\mathbb{R}^n}(\cos(\xi\cdot z)-1) d\nu(z) = -|\xi|^\alpha.
\end{equation*}
In this case, the numerator of (\ref{Levym}) is given by 
\begin{align*}
C_{n,\alpha}\int_{\mathbb{R}^n}(\cos(\xi\cdot z)-1)\varphi(z) d\nu(z) &= C_{n,\alpha} \int_{\mathbb{S}^{n-1}}\varphi(\theta)\int_0^\infty \cos(r\xi\cdot\theta)r^{-1-\alpha} dr d\sigma(\theta)\\
&=C_{n,\alpha}\int_{\mathbb{S}^{n-1}}\varphi(\theta)|\xi\cdot\theta|^\alpha\int_0^\infty \cos(s)s^{-1-\alpha} ds d\sigma(\theta)\\
&=C'_{n,\alpha}\int_{\mathbb{S}^{n-1}}\varphi(\theta)|\xi\cdot\theta|^\alpha d\sigma(\theta).
\end{align*}
Therefore, the corresponding multiplier is given by 

\begin{equation*}
\label{alphamult}m_\alpha(\xi) = \frac{\int_{\mathbb{S}^{n-1}}|\xi\cdot\theta|^\alpha \varphi(\theta)d\sigma(\theta)}{\int_{\mathbb{S}^{n-1}}|\xi\cdot\theta|^\alpha d\sigma(\theta)}.
\end{equation*}
If we set $n=2$ and choose $\varphi(\theta)=e^{-2i\arg{\theta}}$, then it is shown in \cite{BanBog2} that $m_\alpha(\xi) = \frac{\alpha}{\alpha+2}\frac{\bar{\xi}}{\xi}$. Therefore, for all $0<\alpha<2$ and all $f\in L^p(\mathbb{R}^n)$
\begin{equation*}
\label{Balpha} \|Bf\|_p\leq \frac{\alpha+2}{\alpha} (p^*-1)\|f\|_p.
\end{equation*}
Letting $\alpha \nearrow 2$, we recover (\ref{twobound}). 

The condition $0<\alpha<2$ is natural from a probabilistic prospective. Otherwise, the measure $d\nu(z) = \frac{1}{|z|^{n+\alpha}}$ is not a L\'evy measure on $\mathbb{R}^n$. However, for any $r>0$, the multiplier
\begin{equation}
\label{mr}m_r(\xi) = \frac{\int_{\mathbb{S}^{n-1}}|\xi\cdot\theta|^r \varphi(\theta)d\sigma(\theta)}{\int_{\mathbb{S}^{n-1}}|\xi\cdot\theta|^r d\sigma(\theta)}
\end{equation}
satisfies $\|m_r\|_\infty\leq 1$. Therefore, $T_{m_r}$ is a bounded operator on $L^2(\mathbb{R}^n).$ Furthermore, for any $r>0$, if we choose $\varphi(\theta)=e^{-2i\arg{\theta}}$, the formula $T_{m_r}f(x) = \frac{r}{r+2}Bf(x)$ is valid for all $f\in C^\infty(\mathbb{R}^n)$. Therefore, if we could prove conjecture (\ref{Conj}), then letting $r\rightarrow\infty$ it would follow that $\|B\|_p\leq p^*-1$, and therefore the celebrated conjecture of Iwaniec would be proved. %then the conjecture of Iwaniec that $\|B\|_p=p^*-1$ would follow by letting $r\rightarrow \infty$.
 Unfortunately, we are not able to prove conjecture (\ref{Conj}) in its entirety. We are, however, able to show that $m_r$ defined as in (\ref{mr}) is an $L^p$ multiplier for all $1<p<\infty$ and for all $r>0$.

The probabilistic methods used in \cite{BanBog1} and \cite{BanBog2} do not apply when $r\geq 2$.  This leads us to study $T_{m_r}$ through analytic methods. Two tools for doing so are the Marcinkiewicz mutliplier theorem and the H\"ormander-Mikhlin multiplier theorem which we state below for convenience.  For proofs of these results see \cite{Gra} or \cite{Ste2}.

 \begin{theorem} (Marcinkiewicz). \label{Marcinkiewicz}
Let $m\in L^\infty(\mathbb{R}^n)$ with $\|m\|_\infty \leq K$ for some $0<K<\infty$. Supposed that $m(\xi)$ is $n$-times continuously differentiable on the subset of $\mathbb{R}^n$ where none of the $\xi_i$ are zero. For $j\in \mathbb{Z}$, let $I_j$ denote the dyadic interval $(-2^{j+1},-2^{j}]\cup [2^{j},2^{j+1})$. Suppose that for all $1\leq k\leq n$, for all subsets $\{i_1,\ldots,i_k\}$ of $\{1,\ldots,n\}$ of order $k$, and for all integers $l_{i_1},\ldots l_{i_k}$, 
we have that 

\begin{equation}
\label{Marc}\int_{I_{l_{i_1}}}\ldots\int_{I_{l_{i_k}}} |\partial_{i_1}\ldots\partial_{i_k}m(\xi)|d\xi_{i_k}\ldots d\xi_{i_1} \leq K <\infty
\end{equation}
whenever $\xi_j\neq 0$ for all $j\notin\{i_1,\ldots,i_k\}$. Then $m(\xi)$ is an $L^p$ multiplier for all $1<p<\infty$ and

\begin{equation*} 
 \|T_mf\|_p \leq C_n K (p^*-1)^{6n}\|f\|_p \quad\text{for all }f\in L^p(\mathbb{R}^n),
\end{equation*}
where $C_n$ is a constant depending only on $n$.
 \end{theorem}
 
\begin{theorem} (H\"ormander-Mikhlin). Let $n_0 = \lfloor\frac{n}{2}\rfloor+1$, and let $m(\xi)$ be $n_0$-times differentiable on $\mathbb{R}^n\setminus\{0\}$. Suppose there exists $0<K<\infty$ such that $\|m\|_\infty \leq K$ and that also
\begin{equation}
\label{Hormander}\sup_{R>0} R^{-n+2|\beta|} \int_{R<|\xi|<2R} |\partial^\beta m(\xi)|^2 d\xi < K^2
\end{equation}
for all multi-indexes such that $|\beta|\leq n_0$. Then $m(\xi)$ is an $L^p$ multiplier for all $1<p<\infty$ and there exists $C_n$ depending only on $n$ such that
\begin{equation*}
\|T_mf\|_p \leq C_n K (p^*-1)\|f\|_p.
\end{equation*}
 \end{theorem}
 
%
%\begin{theorem} Let $0<r<\infty$, let $\varphi\in L^\infty(\mathbb{S}^{n-1})$ and let 
%\begin{equation}
%m(\xi) = \frac{\int_{S^{n-1}} |\xi\cdot\theta|^r\varphi(\theta) d\sigma(\theta)}{\int_{S^{n-1}} |\xi\cdot\theta|^r d\sigma(\theta)}.
%\end{equation}
%Then the corresponding multiplier operator, $T_m$, is bounded on $L^p(\mathbb{R}^n)$ for all $1<p<\infty$. In particular, there exists a constant $C_{n,p}$ which depends only on $n$ and $p$ such that
%
%\begin{equation}
%\|T_mf\|_p \leq C_{n,p}r^{(n-1)/2}\|f\|_p
%\end{equation}
%for all $f\in L^p(\mathbb{R}^n)$.
%\end{theorem} 

\section{The proof of theorem \ref{main}} \label{proof}
The main idea of the proof is to use a method of rotations to write $T_{m_r}$ as the weighted average of multipliers which can be studied using the Marcinkiewicz multiplier theorem.
\begin{proof}  We first observe (see \cite{Gra} Appendix D p. 443) that

\begin{equation}
\label{Grafak}\int_{\mathbb{S}^{n-1}} |\xi\cdot\theta|^r d\sigma(\theta)  = A_{n,r}|\xi|^r, 
\end{equation}
where $A_{n,r} = \frac{1}{2\pi^{(n-1)/2}}\frac{\Gamma(\frac{1+r}{2})}{\Gamma(\frac{n+r}{2})}$. Therefore,  
\begin{equation}
\label{newm} m_r(\xi) = A_{n,r}^{-1}\int_{\mathbb{S}^{n-1}} \frac{|\xi\cdot\theta|^r}{|\xi|^r} \varphi(\theta) d\sigma(\theta).
\end{equation}

Now for $\theta\in \mathbb{S}^{n-1}$,  we let $m_\theta(\xi) = \frac{|\xi\cdot\theta|^r}{|\xi|^r}$.  Using (\ref{newm}), we may write $T_{m_r}$ as a weighted average of the $T_{m_\theta}$'s. More precisely, we shall prove the following lemma. 

\begin{lemma} For all $f\in C^\infty_0(\mathbb{R}^n)$, 
\begin{equation*} 
T_{m_r}f(x) = A_{n,r}^{-1}\int_{\mathbb{S}^{n-1}}T_{m_\theta} f(x)\varphi(\theta)d\sigma(\theta), 
\end{equation*}
for almost every $x$.
\end{lemma}

\begin{proof} Let $f$ and $g\in C_0^\infty(\mathbb{R}^n)$. Then by Plancherel's theorem, Fubini's theorem, and the Cauchy-Schwarz inequality, 
\begin{align*}
&A_{n,r}^{-1}\int_{\mathbb{R}^n} \int_{\mathbb{S}^{n-1}}T_{m_\theta} f(x)\varphi(\theta)d\sigma(\theta)g(x)dx\\
=& A_{n,r}^{-1}\int_{\mathbb{S}^{n-1}}\varphi(\theta) \int_{\mathbb{R}^n} T_{m_\theta} f(x) g(x) dx d\sigma(\theta)\\
=& A_{n,r}^{-1}\int_{\mathbb{S}^{n-1}}\varphi(\theta) \int_{\mathbb{R}^n} m_\theta(\xi) \widehat{f}(\xi) \bar{\widehat{g}}(\xi) d\xi d\sigma(\theta)\\
=& A_{n,r}^{-1}\int_{\mathbb{R}^n}\int_{\mathbb{S}^{n-1}}  m_{\theta}(\xi)\varphi(\theta)d\sigma(\theta) \widehat{f}(\xi) \bar{\widehat{g}}(\xi) d\xi\\
=& \int_{\mathbb{R}^n} \widehat{T_{m_r} f}(\xi)\bar{\widehat{g}}(\xi) d\xi\\
=& \int_{\mathbb{R}^n} T_{m_r}f(x)g(x)dx.
\end{align*}
\end{proof}
%\noindent The two uses of Fubini's theorem are justified by the Cauchy-Schwarz inequality and the observation that $\|T_\theta f\|_2\leq \|f\|_2.$ 

We will also need to estimate the $L^p$ boundedness of the operators $T_{m_\theta}$.   This is accomplished by the following lemma.  

\begin{lemma} \label{Ttheta} There exist $0<C_{n}<\infty$ such that 
\begin{equation*}
\|T_{m_\theta} f\|_p \leq C_{n} (p^*-1)^{6n} \|f\|_p, 
\end{equation*}
for all $f\in L^p(\mathbb{R}^n)$. $C_{n}$ depends only on $n$ and, in particular, does not depend on $r$ or $\theta$.
\end{lemma}

Before proving lemma \ref{Ttheta}, we will first show how it is used to give a simple proof of Theorem \ref{main}.  By Minkowski's integral inequality,
\begin{align*}
\|T_{m_r}f\|_p &=A_{n,r}^{-1}\left(\int_{\mathbb{R}^n}\left|\int_{\mathbb{S}^{n-1}} \varphi(\theta) T_{m_\theta} f (x) d\sigma(\theta)\right|^pdx\right)^{1/p}\\
&\leq A_{n,r}^{-1}\int_{\mathbb{S}^{n-1}} \left(\int_{\mathbb{R}^n}|\varphi(\theta)|^p|T_{m_\theta} f(x)|^p dx\right)^{1/p} d\sigma(\theta)\\
&= A_{n,r}^{-1}\int_{\mathbb{S}^{n-1}} |\varphi(\theta)|\left(\int_{\mathbb{R}^n}|T_{m_\theta} f(x)|^p dx\right)^{1/p} d\sigma(\theta)\\
&= A_{n,r}^{-1}\int_{\mathbb{S}^{n-1}} \|T_{m_\theta} f\|_p d\sigma(\theta)\\
&\leq A_{n,r}^{-1} C_{n}(p^*-1)^{6n}\omega_{n-1}\|f\|_p,
\end{align*}
where $\omega_{n-1}$ is the surface area of $\mathbb{S}^{n-1}$.
Therefore, theorem \ref{main} is proved.
\end{proof}

We shall now prove lemma \ref{Ttheta} 
\begin{proof}
  For $\theta$ in $\mathbb{S}^{n-1}$, let $R$  be a rotation such that $R\theta = e_1$ and for $f \in L^p$ let $g(x) = f(R^{-1}x)$. Then a simple change for variables shows that $T_{m_\theta} f(x) = T_{m_{e_1}}g(R x)$. Therefore, it suffices to show that 
\begin{equation*}
\label{e1bound}\|T_{m_{e_1}}f\|_p \leq C_{n}(p^*-1)^{6n}\|f\|_p \quad \text{for all }f\in L^p(\mathbb{R}^n).
\end{equation*}

To prove this, we will show that $m_{e_1}$ satisfies the assumptions of theorem \ref{Marcinkiewicz} and that we can take $K$ to be independent of $r$ in (\ref{Marc}). Note that it follows from \cite[p. 110]{Ste2} that for each fixed $r$, $T_{m_{e_1}}$ is a Marcinkiewicz multiplier, but it takes considerably  more work to show that $K$ can be taken to be independent of $r$ in (\ref{Marc}).
%The proof of (\ref{e1bound}) is an exercise in the Marcinkiewicz multiplier theorem which we recall below for convenience. We should note that is is well-known (see for example \cite{Ste2} p.110) that these operators are bounded on $L^p(\mathbb{R}^n)$. The novelty of (\ref{e1bound}) is that the constant $C_{n,p}$ can be take independent of $r$. 
%
%We would like to show that $K$ can be chosen independent of $r$.
$m_{e_1}(\xi)$ is even in each $\xi_i$ so it suffices to restrict attention to the region where all $\xi_i$ are positive. Noting that for all $A_1,\ldots,A_k>0$
\begin{equation*}
\int_{A_1}^{2A_1}\ldots\int_{A_k}^{2A_k}\frac{1}{\xi_{i_1}\xi_{i_2}\ldots\xi_{i_k}}d\xi_{i_k}\ldots d\xi_{i_1} = \log(2)^k,
\end{equation*}
we see that, it suffices to prove there exists $C$ independent of $r$ such that %whenever $\xi_{i_1},\ldots,\xi_{i_k}>0$

\begin{equation*}
  |\partial_{i_1}\ldots\partial_{i_k}m_{e_1}(\xi)| \leq \frac{C}{\xi_{i_1}\xi_{i_2}\ldots\xi_{i_k}}
\end{equation*}
or equivalently that
\begin{equation}
\label{homog} \xi_{i_1}\xi_{i_2}\ldots\xi_{i_k}|\partial_{i_1}\ldots\partial_{i_k}m_{e_1}(\xi)| \leq C.
\end{equation}

The left hand side of (\ref{homog}) is homogeneous of order zero, so it suffices to bound this quantity on the portion of the unit sphere where all $\xi_i\geq0$. To do this, we will make use of two elementary lemma's which involve the use of Lagrange multipliers to bound polynomials on ellipses. 

 \begin{lemma} \label{lagrange1} Let $a,b,c,d>0$. The maximum value of
\begin{equation*}
f(x,y) = x^ay^b 
\end{equation*}
subject to the constraints $cx^2+dy^2=1$, $x,y \geq 0$, is given by 
\begin{equation*}
\frac{\left(\frac{a}{c}\right)^{a/2}\left(\frac{b}{d}\right)^{b/2}}{(a+b)^{(a+b)/2}}.
\end{equation*} 
\end{lemma}

\begin{proof} It is easy to check using the method of Lagrange multipliers to show that $f$ is maximized when 
\begin{equation*}
\label{findx} x^2 = \frac{a}{c(a+b)}\quad \text{and}\quad y^2 = \frac{b}{d(a+b)}.
\end{equation*}
The result follows immediately. 
\end{proof}
%First we note that $f$ cannot attain it's maximum at any point where either $x$ or $y$ is zero. Taking partial derivatives in the $x$ direction we see
%\begin{equation*}
%ax^{a-1}y^b = 2cx\lambda
%\end{equation*}
%which implies 
%\begin{equation*}
%\frac{a}{2c} x^{a-2}y^b = \lambda.
%\end{equation*}
%Likewise, taking the partial derivatives in the $y$ direction we have 
%\begin{equation*}
%bx^ay^{b-1}=2dy\lambda.
%\end{equation*}
%So,
%\begin{equation*}
%\frac{b}{2d}x^ay^{b-2}=\lambda.
%\end{equation*}
%Therefore,
%\begin{equation*}
%\frac{a}{2c} x^{a-2}y^b = \frac{b}{2d}x^ay^{b-2}.
%\end{equation*}
%This implies that 
%\begin{equation*}
%y^2= \frac{bc}{ad}x^2.
%\end{equation*}
%The constraint equation now becomes
%\begin{align*}
%1 &= cx^2+ d\frac{bc}{ad}x^2\\
%&=\frac{c(a+b)}{a}x^2
%\end{align*}
%which implies (\ref{findx}).

\begin{lemma} \label{lagrange2} Let $1<k\leq n$, then 
the maximum value of $f(x,y,z) = (k-1)x^{2k}y^r + (n-k)x^{2k-2}y^r z^2$ subject to the constraint that $g(x,y,z)=(k-1)x^2+y^2+(n-k)z^2 = 1, x,y,z\geq 0$  is 

\begin{equation*}
\frac{(2k)^k}{(k-1)^{k-1}}\left(\frac{r}{2k+r}\right)^{r/2}\frac{1}{(2k+r)^k}.
\end{equation*}
\end{lemma}

\begin{proof}
If $k=n$ then, 
\begin{equation*}
f(x,y,z) = f(x,y)= (n-1)x^{2n}y^r \quad\text{and}\quad g(x,y,z) = g(x,y) = (n-1)x^2+y^2,
\end{equation*} 
so the result follows from lemma \ref{lagrange1}. If $1<k<n$, the method of Lagrange multipliers can be used to show that  at any point at which $f$ achieves a local maximum, $z=0$. Therefore, the result again follows from lemma \ref{lagrange1}.
\end{proof}
Now, in order to verify that $m_r$ satisfies (\ref{homog}), we consider three cases.

\begin{case}\label{nospecialk}
$1\notin \{i_1,\ldots,i_k\}:$
\end{case}
 By direct computation,
\begin{align*}
|\partial_{i_1}\ldots\partial_{i_k}m_{e_1}(\xi)| &= r(r+2)\ldots(r+2k-2)\frac{\xi_1^{r}\xi_{i_1}\ldots\xi_{i_k}}{|\xi|^{r+2k}}.
\end{align*}
Therefore, we need to bound 
\begin{equation*}
 r(r+2)\ldots(r+2k-2)\xi_1^{r}\xi^2_{i_1}\ldots\xi^2_{i_k}
\end{equation*} 
on the portion of the unit sphere where all coordinates are non-negative. By symmetry, it is clear that this last term is maximized when $\xi_{i_1}=\xi_{i_2}=\ldots=\xi_{i_k}$ and $\xi_i=0$ , whenever $i\notin\{i_1,\ldots,i_k,1\}$. Therefore, we are  lead to the two-dimensional optimization problem of maximizing 
\begin{equation*}
f(x,y) = x^{2k}y^r,
\end{equation*}
subject to the constraint that $g(x,y) = kx^2+y^2=1$. By lemma \ref{lagrange1}, the maximal value of $f$ subject to this constraint is less than 
\begin{equation*}
 C_k \left(\frac{1}{2k+r}\right)^k.
\end{equation*}
Therefore, on the unit sphere
\begin{equation*}
r(r+2)\ldots(r+2k-2)\xi_j^{r}\xi^2_{i_1}\ldots\xi^2_{i_k} \leq C_k \frac{r(r+2)\ldots(r+2k-2)}{(2k+r)^k} \leq C_k.
\end{equation*}

\begin{case} \label{onlyspecialk}
$k=1, i_1=1:$ 
\end{case} Differentiating, we see
\begin{align*}
|\xi_1\partial_1 m_{e_1}(\xi)| &= r\frac{\xi_1^r}{|\xi|^{r+2}}(\xi_2^2+\ldots+\xi_n^2), 
\end{align*}
and (\ref{homog}) can be verified by repeating the arguments of case \ref{nospecialk}.
%so it suffices to bound
%\begin{equation*}
%r\xi_1^{r}(\xi_2^2+\ldots+\xi_n^2) 
%\end{equation*}
%on the sphere. This follows from lemma \ref{lagrange1} similarly to case \ref{nospecialk}. %Again, by symmetry, the maximum occurs when $\xi_2=\ldots=\xi_n$, so we can again apply lemma 3 with $f(x,y) = x^r y^2$ and $g(x,y) = x^2+(n-1)y^2$ to see
%%\begin{align*}
%%r\xi_1^{r}(\xi_2^2+\ldots+\xi_n^2) \leq 2\left(\frac{r}{r+2}\right)^{r/2}\frac{r}{r+2} < 2.
%%\end{align*}

\begin{case}  $k>1$ and $1\in \{i_1,\ldots,i_k\}:$ 
\end{case}
 Without loss of generality, we may assume $i_k=1$. Carrying out the computations, we see
\begin{align*}
&|\partial_{i_1}\ldots\partial_{i_{k-1}}\partial_{1} m(\xi)| \\
=&\Bigg{|}\frac{r(r+2)\ldots(r+2k-4)r\xi_{i_1}\ldots\xi_{i_{k-1}}\xi_{1}^{r-1}}{|\xi|^{r+2k-2}}
-  \frac{r(r+2)\ldots(r+2k-2)\xi_{i_1}\ldots\xi_{i_{k-1}}\xi_{1}^{r+1}}{|\xi|^{r+2k}}\Bigg{|}\\
=&\frac{r(r+2)\ldots(r+2k-4)\xi_{i_1}\ldots\xi_{i_{k-1}}\xi_{1}^{r-1}}{|\xi|^{r+2k}}\left|r(\xi_2^2+\xi_{3}^2+\ldots+\xi_n^2)-(2k-2)\xi_1^2\right|.
\end{align*}
%Thus we need to bound
%\begin{equation*}
%(r)(r+2)\ldots(r+2k-4)\xi^2_{i_1}\ldots\xi^2_{i_{k-1}}\xi_{1}^{r}\left(r(\xi_2^2+\ldots+\xi_n^2)-(2k-2)\xi_1^2\right)
%\end{equation*}
%on the unit sphere. 
%It
Therefore, it suffices to show that there exists $C_k$ such that
\begin{equation*}
r(r+2)\ldots(r+2k-4)\xi^2_{i_1}\ldots\xi^2_{i_{k-1}}\xi_1^{r+2} <C_k 
\end{equation*}
and 
\begin{equation*}
r(r+2)\ldots(r+2k-4)r\xi^2_{i_1}\ldots\xi^2_{i_{k-1}}\xi_1^{r}(\xi_2^2+\ldots+\xi_n^2) <C_k,
\end{equation*}
whenever $|\xi|=1$ and all $\xi_{i}\geq 0$.
This can be done by using lemmas \ref{lagrange1} and \ref{lagrange2} in a manner similar to  cases \ref{nospecialk} and \ref{onlyspecialk}.
\end{proof}

\begin{remark}
In the case that $r=2k$ is an even integer, we have that $T_{e_1} =R_1^{2k}$, the $2k-th$ order Riesz transform in direction 1. Dimension free estimates for this operator were obtained by Iwaniec and Martin in \cite{IwaMar} using a method that compared polynomials of the Riesz transforms to polynomials of the complex Riesz transforms and then in turn estimated the complex Riesz transforms by comparing them to the iterated Beurling-Ahlfors transform.

  Identifying $\mathbb{C}^n$ with $\mathbb{R}^{2n}$ the complex Riesz transforms are defined by  
\begin{equation*}
C_j = R_j + i R_{n+j}
\end{equation*}
for $1\leq j\leq n$. For a polynomial $p(x) = \sum_{|\beta|\leq m} c_\beta x^{\beta}$, $p(\mathbf{R})$ and $p(\mathbf{C})$ are defined by 
\begin{equation*}
p(\mathbf{R}) = \sum_{|\beta|\leq m} c_\beta \mathbf{R}^{\beta}
\quad
\text{ and }\quad
p(\mathbf{C}) = \sum_{|\beta|\leq m} c_\beta \mathbf{C}^{\beta},
\end{equation*}
\noindent where $\mathbf{R}^\beta= R_1^{\beta_1}\circ\ldots \circ R_n^{\beta_n}$ and $\mathbf{C}^\beta= C_1^{\beta_1}\circ\ldots \circ C_n^{\beta_n}$. Iwaniec and Martin then show that if $p_{2k}$ is a homogeneous polynomial of degree $2k$ we have that 

\begin{equation*}
\|p_{2k}(\mathbf{R})\|_{L^p(\mathbb{R}^n)\rightarrow L^p(\mathbb{R}^n)} \leq \|p_{2k}(\mathbf{C})\|_{L^p(\mathbb{C}^n)\rightarrow L^p(\mathbb{C}^n)} \leq \frac{2\Gamma(n+k)\|B^{k}\|_p}{k\pi^n\Gamma(k)} \int_{\mathbb{S}^{2n-1}}|p_{2k}(z)| d\sigma(z),
\end{equation*}
where $\|B^{k}\|_p$ is the norm of the $k$-th iterated Beurling-Ahlfors transform on $L^p(\mathbb{C}).$

\noindent Picking $p(x)=x_1^{2k}$ and computing the integral on the right-hand side using the formulas in Appendix D of \cite{Gra}, we see 
\begin{equation*}
\label{RandB}\|R_1^{2k}\|_{L^p(\mathbb{R}^n)\rightarrow L^p(\mathbb{R}^n)} \leq \|B^k\|_p.
\end{equation*}
 The $L^p$ boundedness of $B^k$ was studied by Dragicevic, Petermichl, and Volberg in \cite{Dra2} where they showed that 
\begin{equation*}
C_1 k^{1-2/p^*}p^* \leq \|B^k\|_{p} \leq C_2k^{1-2/p^*}p^*.
\end{equation*}
Combining this with (\ref{RandB}) gives
\begin{equation*}
\|R_1^{2k}\|_{L^p(\mathbb{R}^n)\rightarrow L^p(\mathbb{R}^n)} \leq C_2k^{1-2/p^*}p^*.
\end{equation*}  
Therefore, 
\begin{equation*}
\|T_{m_r}f\|_p \leq C_n\frac{\Gamma(\frac{n+r}{2})}{\Gamma(\frac{n+1}{2})}\left(\frac{r}{2}\right)^{1-2/p^*}p^*\|f\|_p.
\end{equation*}
Like the bound obtained  in theorem \ref{second}, this bound is linear in $p$. Futhermore, with $p$ fixed it has order $r^{(n+1)/2-2/p^*}$ as $r\rightarrow \infty$, which is slightly better than the bound obtained in theorem \ref{second}. However, this bound has the disadvantage of only being valid when $r$ is an even integer whereas the bound obtained in theorem \ref{second} is valid for all sufficiently large $r$.
\end{remark}

\section{The proof of theorem \ref{second}}
\begin{proof}
It is clear that $\|m_r\|_\infty \leq 1$, so
by  (\ref{Hormander}) it suffices to show that 
\begin{equation*}
\left(\sup_{R>0} R^{-n+2|\beta|}\int_{R<|\xi|<2R} |\partial^{\beta}m_r(\xi)|^2 d\xi\right)^{1/2} \leq C_n r^{|\beta|} %\frac{\Gamma(\frac{r+n}{2})}{\Gamma(\frac{r+1}{2})}\frac{\Gamma(\frac{r-n_0+1}{2})}{\Gamma(\frac{r-n_0+n}{2})}
\end{equation*}
for all multi-indexes with $|\beta|\leq n_0$. 
But since $m_r$ is homogeneous of order zero, we can make a change of variables and then use polar coordinates to see that
\begin{align*}
\sup_{R>0} R^{-n+2|\beta|}\int_{R<|\xi|<2R} |\partial^{\beta}m_r(\xi)|^2 d\xi &= \int_{1<|\xi|<2} |\partial^{\beta}m_r(\xi)|^2 d\xi \\
&= \int_1^2 t^{n-1} \int_{\mathbb{S}^{n-1}}|\partial^\beta m_r(t\xi')|^2 d\sigma(\xi) dt\\
&= \int_1^2 t^{n-1-2|\beta|} dt  \int_{\mathbb{S}^{n-1}}|\partial^\beta m_r(\xi')|^2 d\sigma(\xi) \\
&\leq C_n \int_{\mathbb{S}^{n-1}}|\partial^\beta m_r(\xi)|^2 d\sigma(\xi),
\end{align*}
where $\xi'=\frac{\xi}{|\xi|}$.
Therefore, it suffices to  show that for all multi-indexes $\beta$ with $|\beta|\leq n_0$, 
\begin{equation}
\label{Hormpartials}\left(\int_{\mathbb{S}^{n-1}}|\partial^\beta m_r(\xi)|^2 d\sigma(\xi)\right)^{1/2} \leq C_n  r^{|\beta|}.
\end{equation}

\noindent As in (\ref{Grafak}), we see that 
\begin{equation*}
m_r(\xi) = C_n\frac{\Gamma(\frac{r+n}{2})}{\Gamma(\frac{r+1}{2})}n_r(\xi), 
\end{equation*}
where
\begin{equation*}
n_r(\xi)= \int_{\mathbb{S}^{n-1}}\frac{|\xi\cdot\theta|^r}{|\xi|^r} \varphi(\theta) d\sigma(\theta).
\end{equation*}
\noindent We will show that 

\begin{equation*}
\left(\int_{\mathbb{S}^{n-1}}|\partial^\beta n_r(\xi)|^2 d\sigma(\xi)\right)^{1/2} \leq C_n  r^{|\beta|} \frac{\Gamma(\frac{r-n_0+1}{2})}{\Gamma(\frac{r-n_0+n}{2})}, 
\end{equation*} 
\noindent and so (\ref{Hormpartials}) will follow by observing that Sterling's formula implies that there exists $C_n$ such that for all $r\geq n_0$

\begin{equation*}
\frac{\Gamma(\frac{r+n}{2})}{\Gamma(\frac{r+1}{2})}\frac{\Gamma(\frac{r-n_0+1}{2})}{\Gamma(\frac{r-n_0+n}{2})} \leq C_n.
\end{equation*} 

For all $\theta\in \mathbb{S}^{n-1}$, let $m_\theta(\xi) = \frac{|\xi\cdot\theta|^r}{|\xi|^r}$ so that
\begin{equation*}
\partial^\beta n_r(\xi)  = \int_{\mathbb{S}^{n-1}} \partial^\beta m_\theta(\xi)\varphi(\theta) d\sigma(\theta).
\end{equation*}

\noindent We note that it suffices to show that for all $|\beta|\leq n_0$,
\begin{equation}
\label{partial} |\partial^\beta m_\theta(\xi)| \leq C_n r^{|\beta|} |\xi\cdot\theta|^{r-n_0}.
\end{equation}
For then we see that 
\begin{align*}
\left(\int_{\mathbb{S}^{n-1}}|\partial^\beta n_r(\xi)|^2 d\sigma(\xi)\right)^{1/2} &\leq C_n r^{|\beta|}\left(\int_{\mathbb{S}^{n-1}}\left(\int_{\mathbb{S}^{n-1}}|\partial^\beta m_\theta(\xi)|d\sigma(\theta)\right)^2d\sigma(\xi)\right)^{1/2}\\
&\leq  C_n r^{|\beta|}\left(\int_{\mathbb{S}^{n-1}}\left(\int_{\mathbb{S}^{n-1}}|\xi\cdot\theta|^{r-n_0}d\sigma(\theta)\right)^2d\sigma(\xi)\right)^{1/2}\\
&=C_n r^{|\beta|}\frac{\Gamma(\frac{r-n_0+1}{2})}{\Gamma(\frac{r-n_0+r}{2})}.
\end{align*}

Let $g_\theta(\xi) = |\xi\cdot\theta|^r$ and $h(\xi)=|\xi|^{-r}$ so that $m_\theta(\xi) = g_\theta(\xi)h(\xi)$. By Leibniz's rule 
\begin{align*}
|\partial^\beta m_\theta(\xi)| &= \left| \sum_{\gamma\leq \beta} {\beta \choose \gamma} \partial^\gamma g_\theta(\xi)\partial^\delta h(\xi) \right|\\
&\leq C_n \sum_{\gamma\leq\beta} |\partial^\gamma g_\theta(\xi)||\partial^\delta h(\xi)|,
\end{align*}
where $\delta = \beta-\gamma$.

Letting $\gamma = (\gamma_1,\ldots,\gamma_i)$ and $\delta=(\delta_1,\ldots,\delta_j)$, we see that when $|\theta|=|\xi|=1$ 
\begin{align}
|\partial^\gamma g_\theta(\xi)| &= r(r-1)\ldots (r-i+1)|\xi \cdot \theta|^{(r-i)} \left|\theta_{\gamma_1}\ldots\theta_{\gamma_i}\right| \nonumber\\
\label{gbound}&\leq r^i |\xi\cdot\theta|^{r-n_0}
\end{align}
and
\begin{align}
\label{hbound}|\partial_\delta h(\xi)| = r(r+1)\ldots (r+j-1)|\xi|^{-r-2j}\left|\xi_{\delta_1}\ldots\xi_{\delta_j}\right| \leq C_n r^j.
\end{align}
(\ref{partial}) follows immediately which completes the proof. 
\end{proof}

\begin{remark} \label{McCo}
If we inspect the proof of theorem \ref{second}, we will see that if $r>n+1$, it follows from (\ref{gbound}) and (\ref{hbound}), that $m_r$ is multiplier which satisfies the estimate 
\begin{equation}
|\xi|^{|\beta|} |\partial^\beta m_r(\xi)| \leq C_r
\end{equation}
for all multi-indexes with $|\beta|\leq n+1$. Therefore, by a result of McConnell \cite{McC}, $m_r$ may be obtained using martingale transforms with respect to a Cauchy process.
\end{remark}

\section{The Method of Rotation for other L\'evy Multipliers}

We have seen that the L\'evy multipliers which arise from martingale transforms with respect to $\alpha$-stable processes can be studied analytically using the method of rotations. This approach has the disadvantage that it does not allow us to obtain as good of constants as those that are obtained through probabilistic methods. However, it has the advantage of allowing us to remove the restriction that $\alpha<2$ and thereby obtain a larger class of operators which are bounded on $L^p(\mathbb{R}^n)$ for $1<p<\infty$. It is natural to wonder if this method  can be applied to study the multipliers which arise from other L\'evy processes and if so will it again let us remove restrictions on any relevant parameters.

 Let $(X_t)_{t\geq 0}$ be a L\'evy process whose L\'evy measure $\nu$ is rotationally-symmetric and absolutely continuous with respect to the Lebesgue measure.  Write $\nu$ in polar coordinates as $d\nu = v(r)drd\sigma(\theta)$ for some function $v(r)$. Let $\varphi$ be a bounded function on $\mathbb{R}^n$ that is homogeneous of order zero,  and consider the multiplier given by 

\begin{equation*}
\label{jumpmult}m_\nu(\xi) = \frac{\int_{\mathbb{R}^n} (\cos(\xi\cdot z)-1)\varphi(z)d\nu(z)}{\int_{\mathbb{R}^n} (\cos(\xi\cdot z)-1)d\nu(z)}.
\end{equation*}
Let $\rho(\xi)$ be the L\'evy exponent corresponding to the L\'evy triple $(0,0,\nu)$. Since the $\nu$ is symmetric, $\rho(\xi)$ is real, and therefore
 \begin{equation}
\label{real}\int_{\mathbb{R}^n} (\cos(\xi\cdot z)-1)d\nu(z)=\rho(\xi).
\end{equation}
To examine the numerator define $L:\mathbb{R}\rightarrow\mathbb{R}$ by 
\begin{equation}
\label{L}L(x)=\int_0^\infty (\cos(rx)-1) v(r)dr.
\end{equation}
Then, we have that
\begin{equation}
\label{Lrep} \int_{\mathbb{R}^n} (\cos(\xi\cdot z)-1)\varphi(z)d\nu(z) = \int_{\mathbb{S}^{n-1}} L(\xi\cdot\theta)\varphi(\theta)d\sigma(\theta).
\end{equation} 
Therefore, combining (\ref{real}) and (\ref{Lrep}) we see that the multiplier which arises as the projection of martingale transforms with respect to $X_t$ is given by
\begin{equation*} 
m_\nu(\xi) = \int_{\mathbb{S}^{n-1}} \frac{L(\xi\cdot\theta)}{\rho(\xi)} \varphi(\theta) d\sigma(\theta).
\end{equation*}
Similarly to section \ref{proof}, we set $m_\theta(\xi) = \frac{L(\xi\cdot \theta)}{\rho(\xi)}$ so that 
\begin{equation*}
m_{\nu}(\xi) = \int_{\mathbb{S}^{n-1}} m_\theta(\xi) \varphi(\theta)d\sigma(\theta).
 \end{equation*} 
 Then repeating the arguments of section \ref{proof}, we see that if $T_{m_{e_1}}$ is bounded on $L^p(\mathbb{R}^n)$, then $T_{m_r}$ is bound on $L^p(\mathbb{R}^n)$.
 
More generally, we have the following corollary.

\begin{corollary}\label{generalL} For any function $L:\mathbb{R}\rightarrow \mathbb{R}$, let $AL(\xi)=\int_{\mathbb{S}^{n-1}} L(\xi\cdot\theta) d\sigma(\theta)$. If 
$m_{e_1}(\xi) = \frac{L(\xi_1)}{AL(\xi)}$ is an $L^p$ multiplier for some $1<p<\infty$, then for all $\varphi\in L^\infty(\mathbb{S}^{n-1})$.
\begin{equation*}
m_L(\xi) = \frac{\int_{\mathbb{S}^{n-1}}L(\xi\cdot\theta)\varphi(\theta)d\sigma(\theta)}{\int_{\mathbb{S}^{n-1}}L(\xi\cdot\theta)d\sigma(\theta)}.
\end{equation*}
is also an $L^p$ multiplier. In particular, if for some $C_{n,p}>0$,
\begin{equation*}
\|T_{m_{e_1}}f\|_p \leq C_{n,p}\|f\|_p \quad \text{for all }f\in L^p,
 \end{equation*} 
 then 
 \begin{equation*}
 \|T_{m_L}f\|_p \leq \omega_{n-1}C_{n,p}\|f\|_p \quad \text{for all }f\in L^p.
 \end{equation*}
\end{corollary}

Consider now, for $0<\beta<\alpha<2$, the so-called ``mixed-stable'' process defined by, $Z_t = X_t + a Y_t$ where $X_t$ is a rotationally-symmetric $\alpha$-stable process, $Y_t$ is an independent rotationally symmetric $\beta$-stable process, and $a>0$. $Z_t$ is a L\'evy process with exponent $\rho(\xi)= -(|\xi|^\alpha+a^\beta|\xi|^\beta)$ and L\'evy measure
\begin{equation*}
d\nu(z) = (C_{n,\alpha}r^{-1-\alpha}+C_{n,\beta}a^\beta r^{-1-\beta})drd\sigma(\theta).
\end{equation*} 
In this case, by an argument similar to the $\alpha$-stable case, the corresponding multiplier  
is given by

\begin{equation*}
m_{\alpha,\beta}(\xi) = \frac{\int_{\mathbb{S}^{n-1}} (C_{n,\alpha}|\xi\cdot\theta|^\alpha + C_{n,\beta,a}|\xi\cdot\theta|^\beta) \varphi(\theta)d\sigma(\theta)}{\int_{\mathbb{S}^{n-1}} (C_{n,\alpha}|\xi\cdot\theta|^\alpha + C_{n,\beta,a}|\xi\cdot\theta|^\beta) d\sigma(\theta)}.
\end{equation*}
It is already known that $m_{\alpha,\beta}$ is an $L^p$ multiplier for $1<p<\infty$ by the results of \cite{BanBog1} and \cite{BanBog2}. However, the method of rotations allows us to to remove the restriction that $0<\beta<\alpha<2.$ More precisely, we can prove the following.

\begin{corollary} Let $0<r<s<\infty$, let $C_r,C_s>0$, and let $\varphi\in L^\infty(\mathbb{R}^n)$. Then $m_{r,s}$ defined by 
\begin{equation*}
m_{r,s}(\xi) = \frac{\int_{\mathbb{S}^{n-1}} (C_r|\xi\cdot\theta|^r + C_s|\xi\cdot\theta|^s) \varphi(\theta)d\sigma(\theta)}{\int_{\mathbb{S}^{n-1}} (C_r|\xi\cdot\theta|^r+ C_s|\xi\cdot\theta|^s)d\sigma(\theta) }.
\end{equation*}
is an $L^p$ multiplier, for all $1<p<\infty$ and 
\begin{equation*}
\|T_{m_{r,s}} f\|_p \leq C_{n,r,s} (p^*-1)^{6n} \|f\|_p \quad \text{for all } f\in L^p(\mathbb{R}^n).
\end{equation*}
\end{corollary}
\begin{proof}
As in the proof of theorem \ref{main}, the integral in the denominator can be computed directly and 
\begin{equation*}
\int_{\mathbb{S}^{n-1}} (C_r|\xi\cdot\theta|^r+ C_s|\xi\cdot\theta|^s)d\sigma(\theta) = C'_r|\xi|^r + C'_s|\xi|^s.
\end{equation*}
Therefore, in light of corollary \ref{generalL} it suffices to show that 
\begin{equation*}
m_{e_1}(\xi) = \frac{C_r|\xi_1|^r+C_s|\xi_1|^s}{C'_r|\xi|^r+C'_s|\xi|^s}
\end{equation*}
is an Marcinkiewicz multiplier. As in the proof of lemma \ref{Ttheta}, we restrict attention to the region where all $\xi_i$ are non-negative, and check that $m_{e_1}$ satisfies (\ref{homog}).  We already know that $\frac{|\xi_1|^r}{|\xi|^r}$ satisfies (\ref{homog}) so it suffices to show that 
\begin{equation*}
n(\xi) = \frac{1+a|\xi_1|^t}{b+c|\xi|^t}
\end{equation*}
satisfies (\ref{homog}) for all $a,b,c,t>0$ since it is easy to check using Leibniz's rule that the product of two multipliers which satisfy (\ref{homog}) is again a multiplier satisfying (\ref{homog}). 

%$n$ can readily be checked to be a Marcinkiewicz multiplier.

% Since $n$ is even, it again suffices to show that 
%\begin{equation*}
%|\xi_{i_1}\ldots\xi_{i_k}\partial_{i_1}\ldots\partial_{i_k}n(\xi)| <  C
%\end{equation*}
%when $\xi_1,\ldots,\xi_n>0$ for all $1\leq k \leq n$, $\{i_1,\ldots,i_k\}\subset\{1,\ldots,n\}$.

%By symmetry it suffices to show
%\begin{equation*}
%|\xi_2\ldots\xi_k \partial_2\ldots \partial_kn(\xi)|<C \quad \text{ and }|\xi_1\ldots\xi_k \partial_1\ldots \partial_kn(\xi)|<C
%\end{equation*}
%for $1\leq k\leq n$.

Applying Fa\'a di Bruno's formula to the function $g(h(\xi))$, where $h(\xi) =|\xi|^2$ and $g(x) = \frac{1}{b+cx^{t/2}}$, we see that $\partial_{i_1}\ldots\partial_{i_k} \frac{1}{b+c|\xi|^t}$ is a finite linear combination of terms of the form
\begin{equation*}
\left(\frac{|\xi|^t}{b+c|\xi|^t}\right)^i \frac{\xi_{i_1}\ldots\xi_{i_k}}{|\xi|^{2k}}\frac{1}{b+c|\xi|^t}, \quad 0\leq i \leq k.
\end{equation*}
(\ref{homog}) then follows easily which completes the proof.
\end{proof}

Another example of a L\'evy multipliers which can be studied using the method of rotations arises from the so-called relativistic $\alpha$-stable process.
For $0<\alpha<2$, $M>0$, there exists a L\'evy process, $(X_t)_{t\geq 0}$ with symbol $\rho(\xi) = (|\xi|^2+M^{2/\alpha})^{\alpha/2} - M$ and infinitesimal generator
\begin{equation*}
M-(-\Delta +M^{2/\alpha})^{\alpha/2}.
\end{equation*}
When $\alpha=1$, this operator reduces to free-relativistic Hamiltonian which has been intensely studied because of its applications to relativistic quantum mechanics. For further background information on this process, we refer the reader to \cite{Chen}, \cite{BanYol}, and the references provided in therein. 

Here we will show that the multipliers which arise from taking the projections of martingale transforms with respect to $X_t$ can be studied using the method of rotations. Unfortunately, unlike in the case of the mixed stable processes, the fact that $0<\alpha<2$ will play a crucial role in the proof. Therefore, we will not be able to remove that restriction and obtain a larger class of operators.

\begin{corollary}\label{relat}
Let $0<\alpha<2$, $M>0$, and $\varphi\in L^\infty(\mathbb{R}^{n})$ homogeneous of order zero. Let $d\nu(z)=r^{-1-\alpha}\phi(r)drd\theta$ be the L\'evy measure corresponding to the relativistic $\alpha$-stable process with mass $M$ and let $L$ be defined as in (\ref{L}). Then $\frac{L(\xi_1)}{\rho(\xi)}$ is a Marcinkiewicz multiplier and therefore, by corollary \ref{generalL}  
\begin{equation*}
m_\nu=
\frac{\int_{\mathbb{R}^{n}} (1-\cos(\xi\cdot\theta))\varphi(\theta)d\nu(z)}{\int_{\mathbb{R}^{n}} (1-\cos(\xi\cdot\theta))d\nu(z)}
\end{equation*}
is an $L^p$ multiplier and 
\begin{equation*}
\|T_{m_\nu}f\|_p\leq C_{n,\alpha}(p^*-1)^{6n}\|f\|_p.
\end{equation*}
\end{corollary}

\noindent This is of course a weaker version of results already proven in \cite{BanBog1} and \cite{BanBog2}, but nevertheless, it is interesting to observe that this result can also be obtained analytically. 

\begin{proof}
%We first claim that  for all $0\leq k\leq n-1$, whenever $i_1,\ldots, i_k$ are distinct, $\partial_{i_1}\ldots\partial_{i_k} \frac{1}{\rho(\xi)}$ is a finite linear combination of terms with the form
In \cite{Chen}, it is shown that the L\'evy measure corresponding to $X_t$ can be written in polar coordinates by

\begin{equation*}
d\nu(z) = r^{-1-\alpha}\phi(r) drd\sigma(\theta)
\end{equation*}
where $\phi(r)$ is a bounded positive function that that satisfies

\begin{equation}
\label{phiest}\phi(r) \leq  Ce^{-r}r^{(n+\alpha-1)/2}
\end{equation}
when $r\geq 1$.

Now, by Fa\'a di Bruno's formula, $\partial_{i_1}\ldots\partial_{i_k} \frac{1}{\rho(\xi)}$ is a finite linear combination of terms with the form
\begin{equation}
\label{oneoverrho} \frac{\xi_{i_1}\ldots\xi_{i_k}(|\xi|^2+M^{2/\alpha})^{\frac{\alpha}{2}j-k}}{((|\xi|^2+M^{2/\alpha})^{\frac{\alpha}{2}}-M)^{j+1}}, \quad 0\leq j\leq k.
\end{equation}
%Let $g(\xi) = |\xi|^2+m^{2/\alpha}$ and $h(x) = \frac{1}{x^{\alpha/2}-m}$ so that $\frac{1}{\rho(\xi)} = h(g(\xi))$. We can check by induction that for $0\leq k \leq n-1$, $h^{(k)}$ is  a finite linear combination of terms of the form
%\begin{equation*}
%x^{\frac{\alpha}{2}j-k}(x^{\frac{\alpha}{2}}-m)^{-j-1}, \quad 0\leq j\leq k.
%\end{equation*}
%If $i\neq j$, $\partial_i\partial_j g(\xi)=0$. Therefore, (\ref{oneoverrho}) follows from Faa di Bruno's formula.
Therefore, we see that $\frac{1}{\rho(\xi)}$ is infinitely differentiable on $\mathbb{R}^n\setminus\{0\}$ and %it follows from (\ref{oneoverrho}) that 
\begin{equation*}
|\partial_{i_1}\ldots\partial_{i_k}\frac{1}{\rho(\xi)}| \leq O\left(\frac{1}{|\xi|^{\alpha+k}}\right)\quad\text{as }|\xi|\rightarrow\infty.
\end{equation*}
Near $0$,  each term in (\ref{oneoverrho}) is bounded above by

\begin{align*}
%|\partial_{i_1}\ldots\partial_{i_k}\frac{1}{\rho(\xi)}| &\leq 
&C_{M,n,\alpha} \frac{1}{(|\xi|^2+M^{2/\alpha})^{\frac{\alpha}{2}}-M)^{j+1}} \\
&\leq  C_{M,n,\alpha} \frac{1}{(|\xi|^2+M^{2/\alpha})^{\frac{\alpha}{2}}-M)}\leq O\left(\frac{1}{|\xi|^{2}}\right)\quad\text{as }|\xi|\rightarrow 0.
\end{align*}

It is easy to check using the dominated convergence theorem, the mean value theorem and the fact that $r^{k-\alpha}\phi(r)$ is integrable on $(0,\infty)$  for all $k\geq 1$, that $L$ is infinitely differentiable on $(0,\infty)$. % For example, for $0<h<|x|$, there exists $\bar{h}<h$ such that
%\begin{align}
%\frac{f(x+h)-f(x)}{h} = \int_0^\infty \sin(r(x+\bar{h}))r^{-\alpha} \phi(r)dr \rightarrow \int_0^\infty \sin(rx)r^{-\alpha} \phi(r)dr
%\end{align}
%where we use the fact that $|\sin(r(x+\bar{h}))|\leq r|x+h|\leq 2r|x|$ in order to apply the dominated convergence theorem.
Therefore, in order to show that $\frac{L(\xi_1)}{\rho(\xi)}$ is a Marcinkiewicz multiplier it suffices to show that
\begin{align}
\label{Lest1}|L(\xi)|\leq C_\alpha\min\{|\xi|^\alpha,|\xi|^2\} 
\end{align}
and
\begin{align}
\label{Lest2}|L'(\xi)|\leq C_\alpha\min\{ |\xi|^{\alpha-1},|\xi|\}.
\end{align}
For then it will follow that $\frac{L(\xi_1)}{\rho(\xi)}$ satisfies $(\ref{homog})$ since
%For then it will follow by Leibniz's rule that 
\begin{equation*}
\left|\xi_{i_1}\ldots\xi_{i_k}\partial_{i_1}\ldots\partial_{i_k} \frac{L(\xi_1)}{\rho(\xi)}\right|
\end{equation*} 
%is bounded on the region of $\mathbb{R}^n$ where all coordinates are positive since it
 is a continuous function on $\mathbb{R}^n\setminus\{0\}$ which is bounded near the origin and as $|\xi|\rightarrow\infty$. %It therefore will follow that $\frac{L(\xi_1)}{\rho(\xi)}$ is a Marcinkiewicz multiplier by arguments similar to those used in the proof of lemma \ref{Ttheta}

Making a change of variables, we see that  
\begin{align*}
|L(x)| &= \left|\int_0^\infty (\cos(rx)-1) r^{-1-\alpha}\phi(r)dr\right|\\
&= |x|^\alpha \left|\int_0^\infty (\cos(s)-1) s^{-1-\alpha}\phi\left(\frac{s}{|x|}\right) ds\right| \leq C_\alpha|x|^\alpha,
\end{align*}
where the last inequality uses the boundedness of $\phi$. On the other hand we can use the inequality $|\cos(x)-1|\leq x^2$, along with (\ref{phiest}) and the boundedness of $\phi$ to see that 
\begin{align*}
|L(x)| &= \left|\int_0^\infty (\cos(rx)-1) r^{-1-\alpha})\phi(r)dr\right|\\
 &\leq |x|^2 \left|\int_0^\infty r^{1-\alpha}\phi(r) dr\right| \leq C_\alpha|x|^2.
\end{align*}
This proves (\ref{Lest1}).  Note that  the fact that $0<\alpha<2$ is needed in order for this integral to converge. 
 
 To prove (\ref{Lest2}) observe that 
 \begin{equation*}
 L'(x) = \int_0^\infty \sin(rx) r^{-\alpha} \phi(r) dr.
 \end{equation*}
Using the fact that $|\sin(x)|\leq |x|,$ it follows that $|L'(x)|\leq C_\alpha|x|$ by mimicing the above arguments. To obtain the other part of (\ref{Lest2}) we a change of variables, and use the fact that $\varphi$ is decreasing to see
\begin{align*}
|L'(x)| &= |x|^{\alpha-1} \left|\int_0^\infty \sin(t)t^{-\alpha}\varphi\left(\frac{t}{x}\right) dt\right|\\
&= |x|^{\alpha-1} \sum_{n=0}^\infty (-1)^n \int_{n\pi}^{(n+1)\pi} \left|\sin(t)t^{-\alpha}\varphi\left(\frac{t}{x}\right)\right| dt\\
&\leq |x|^{\alpha-1} \left|\int_0^\pi \sin(t)t^{-\alpha}\varphi\left(\frac{t}{x}\right) dt\right|\\
&\leq C_\alpha |x|^{\alpha-1}.
\end{align*}
This completes the proof of corollary (\ref{relat}).
\end{proof}
 
\noindent\textbf{Acknowledgments.} I would like to thank my Ph.D. advisor, Rodrigo Ba\~nuelos, for being an invaluable source of advice throughout the process of writing this paper.


\begin{thebibliography}{50} 
\bibitem{AppBan}{D. Applebaum and R. Ba\~nuleos}, {\it Martingale transforms and L\'evy processes on Lie groups}, Indiana University Mathematics Journal \textbf{63} (2014), 1109-1138.

\bibitem{AstIwaMar}{K. Astala, T. Iwaniec and G. Martin,} {\it Elliptic Partial Differential Equations and Quasiconformal
Mappings in the Plane,}  Princeton University Press, 2009.
\bibitem{Ban1}{R. Ba\~nuelos}, {\it The foundational inequalities of D.L. Burkholder and some of their ramifications}, Illinois J. Math. \textbf{54} (2010) 789-868.


\bibitem{Ban2}{R. Ba\~nuelos}, {\it Martingale transforms and related singular integrals}, Transactions of the American Math. Soc. \textbf{293} no. 2, (1986) 547-563.

%\bibitem{BanBau} R. Ba\~nuelos and F. Baudoin, {\it Martingale transforms and their projection operators on manifolds,} Potential Analysis, {\bf 38} (2013) 1071-1089. 

\bibitem{BanBog1} R. Ba\~nuelos and K. Bogdan,  {\it L\'evy processes and Fourier multipliers}, J. Funct. Anal. {\bf 250} (2007) 197-213.

\bibitem{BanBog2} R. Ba\~nuelos, A. Bielaszewski, and K. Bogdan {\it Fourier multipliers for non-symmetric L\'evy processes}, {\it Marcinkiewicz Centenary Volume}, volume 95, (2011) 9-25.

\bibitem{BanJan}{R. Banuelos and P. Janakiraman}, {\it $L^p$-bounds for the Beurling-Ahlfors transform}, Trans. Amer. Math. Soc. \textbf{360} (2008), 3604-3612.


\bibitem{BanMen}{R. Ba\~nuelos and P. M\'endez-Hern\'andez}, {\it Space-time Brownian motion and the Beurling-Ahlfors transform}, Indiana University Math J. \textbf{52} (2003),  981-990.

\bibitem{BanWan}{R. Ba\~nuelos and G. Wang}, {\it Sharp inequalities for martingales with applications to the Beurling-Ahlfors and Riesz Transforms}, Duke Math. J. \textbf{80} no. 3, (1995) 575-600.

%\bibitem{BanOse} R. Ba\~nuelos and A. Os\c ekowski, {\it Sharp martingale inequalities and applications to Riesz transforms on manifolds, Lie groups and Gauss space,} (preprint) 

\bibitem{BanYol} R. Ba\~nuelos and S. Yolcu, {\it Heat trace of non-local operators,} J. Long. Math. Soc. (2) \textbf{87} no. 3 (2014) 304-318.

%\bibitem{Ber} J. Bertoin, {\it L\'evy Processes}, Cambridge University Press, Cambridge 1996.

%\bibitem{Blu} R.M. Blumenthal, R.K. Getoor, {\it Markov Processes and Potential Theory}, Pure Appl. Math, Academic Press, New York 1968.

%\bibitem{Bog} K. Bogdan, A St\'os, P. Sztonyk, {\it Harnack inequality for symmetric stable processes on d-sets}, Studia Math. \textbf{158} (2003), 163-198

\bibitem{BorJanVol} A.  Borichev, P. Janakiraman, A. Volberg, {\it Subordination by orthogonal martingales in $L^{p}$ and zeros of Laguerre polynomials}, Duke Math. J. Volume 162, No.5 (2013) 889-924. 

%\bibitem{Bur1}{D. L. Burkholder}, {\it Martingale transforms}, Ann. Math. Statist. \textbf{37} (1966), 1494-1504.

%\bibitem{Bur2} D. L. Burkholder, {\it  Boundary value problems and sharp inequalities for martingale transforms}, 
%Ann. Probab. {\bf 12} (1984), 647-702.

\bibitem{Bur3}{D. L. Burkholder}, {\it Sharp inequalities for martingales and stochastic integrals}, Asterique \textbf{157-158} (1988) 75-94.

\bibitem{Chen}{Z.Q. Chen, P. Kim, and R. Song}, {\it Sharp heat kernel estimates for relativistic stable processes on open sets}, Annals of Probability, \textbf{40}(1) 213-244 (2012).  

%\bibitem{Dra}{O. Dragi\v{c}evi\'c, L. Grafakos, M. Pereya, and S. Petermichl} {\it Extrapolation and sharp norm estimates for the classical operators on weighted Lebesgue spaces.}, Publ. Mat. \textbf{49}(1) (2005) 73-91.

\bibitem{Dra2}{O. Dragi\v{c}evi\'c, S. Petermichl, and A. Volberg}{\it A rotation method which gives linear $L^p$ estimates for powers of the Ahlfors-Beurling operator} J. Math, Pures Appl. \textbf{86} (2006) no.6 494-509.

%\bibitem{Dur}{R. Durrett}, {\it Brownian Motion and Martingales in Analysis}, Wadsworth, Belmont, CA (1984).

%\bibitem{GeiMonSmi}{S. Geiss, S. Mongomery-Smith  and E. Saksman, {\it On singular integral and martingale transforms},
%Trans. Amer. Math. Soc., {\bf 362} (2010),  555-575.}

\bibitem{Gra} L. Grafakos, {\it Classical and Modern Fourier Analysis}, Pearson Education, Inc.  New Jersey, 2004. 


\bibitem{GunVar}{R. Gundy and N. Varopoulos, {\it Les transformations de Riesz et les integrales stochastiques}, C. R. Acad. Sci. Paris Ser. A \textbf{289} (1979), 13-16}.

%\bibitem{Hyt}{T. P. Hyt\"onen, {\it The sharp weighted bound for general Calder\'on-Zygmund operators}, Annals of Math.  \textbf{175} Issue 3 (2012), 1473-1506}.

\bibitem{IwaMar}{T. Iwaniec and G. Martin}, {\it The Beurling-Ahlfors transform in $R^n$ and related singular integrals}, J. Reine Agnew. Math. \textbf{473} (1996), 25-27.

\bibitem{Iwa}{T. Iwaniec}, {\it Extremal inequalities in Sobolev spaces and quasiconfromal mappings}, Z. Anal, Anwendungen 
\textbf{6} (1982), 1-16.


%\bibitem{Jan}{P. Janakiraman}, {\it Weak-type estimates for singular integrals and the Riesz transforms}, Indiana University Math.  J. \textbf{53}, no. 2 (2004) 533-555. 

%\bibitem{Kim}{P. Kim, R. Song, Z. Vondracek} {\it Potential Theory of Subordinate Brownian Motion Revisited},Stochastic Analysis and Applications to Finance--Essays in Honour of Jia-an Yan, World Scientific (2012) 243-290.

\bibitem{Leh}{O. Lehto}, {\it Remarks on the integrability of the derivative of quasiconformal mappings}, Ann. Acad. Sci. Fenn. Ser. A. Mat. \textbf{371} (1965), 3-8.

\bibitem{McC}{T. McConnell}, {\it On Fourier multiplier transformations of Banach-valued functions}, Trans. Amer. Math. Soc. \textbf{285} (1984), 739-757.

\bibitem{NazVol} F. Nazarov and  A. Volberg, {\it  Heat extension of the Beurling operator and estimates for its norm,} St. Petersburg Math. J. {\bf 15}, (2004), 563-573.


%\bibitem{Ose}{A. Os\c ekowski, {\it Sharp logarithmic inequalities for Riesz transforms}, { J. Funct. Anal.} \textbf{263} (2012), 89-108.}

\bibitem{Perl}{M. Perlmutter, {\it On a Class of Calder\'on-Zygmund Operators Arising from Projections of Martingale Transforms}, {Potential Analysis.} \textbf{42}(2015). 383-401.}

\bibitem{Protter}{P.E. Protter, {\it Stochastic Integration and Differential Equations}, second ed., Stoch. Model. Appl. Probab., vol. \textbf{21}, Springer-Verlag, Berlin, 2004.}

%\bibitem{Sato} K. Sato, {\it L\'evy Processes and Infinitely Divisible Distributions},  Cambridge University Press, Cambridge 1968.

%\bibitem{Ste1}{E. M. Stein}, {\it Some results in harmonic analysis in $R^n$ for $n\rightarrow \infty$}, Bull. Amer. Math. Soc. (N.S.) \textbf{9} (1983) 71-73.

\bibitem{Ste2}{E. M. Stein}, {\it Singular Integrals and Differentiability Properties of Functions}, Princeton University Press, Princeton (1979).

%\bibitem{Ryznar}{M. Ryznar}, {\it Estimates of Green function for relativistic $\alpha$-stable processes}, Potential Analysis \textbf{17} 1-23 (2002). 








\end{thebibliography}
\end{document}